\documentclass[10pt]{amsart}

\usepackage{latexsym}
\usepackage{amsmath}
\usepackage{amsfonts}
\usepackage{amssymb}
\usepackage{mathrsfs}
\usepackage{bm}
\usepackage{comment} 
\usepackage{accents}
\usepackage[mathscr]{euscript}

\usepackage[usenames,dvipsnames,svgnames,table]{xcolor}

\usepackage{tikz}
\usepackage{tikz-cd}
\usetikzlibrary{matrix,arrows}
\usepackage{scalefnt}

\newcommand{\co}{\mathcal{O}}
\newcommand{\cl}{\mathcal{L}}
\newcommand{\cn}{\mathcal{N}}
\newcommand{\cc}{\mathcal{C}}

\newtheorem{theorem}{Theorem}[section]
\newtheorem{lemma}[theorem]{Lemma}
\newtheorem{corollary}[theorem]{Corollary}
\newtheorem{proposition}[theorem]{Proposition}
\newtheorem{noname*}[theorem]{}

\theoremstyle{definition}
\newtheorem{definition}[theorem]{Definition}

\theoremstyle{remark}
\newtheorem{remark}[theorem]{Remark}

\numberwithin{equation}{section}

\DeclareMathOperator{\Hom}{Hom}
\DeclareMathOperator{\dgHom}{\mathcal{H}om}

\DeclareMathOperator{\Ext}{Ext}

\DeclareMathOperator{\dm}{D^{\textup{b}}_{\textup{m}}}

\newcommand{\dbl}{\mathrm{D}^{\mathrm{b}}_L}
\newcommand{\dbg}{\mathrm{D}^{\mathrm{b}}_G}

\DeclareMathOperator{\vd}{\mathbb{D}}

\DeclareMathOperator{\rep}{Rep}
\DeclareMathOperator{\irr}{Irr}
\DeclareMathOperator{\GL}{GL}

\DeclareMathOperator{\ql}{\bar{\mathbb{Q}}_\ell}

\DeclareMathOperator{\cH}{H}

\DeclareMathOperator{\ic}{IC}
\DeclareMathOperator{\pt}{pt}
\DeclareMathOperator{\perv}{\mathcal{P}erv}

\DeclareMathOperator{\Lie}{Lie}

\DeclareMathOperator{\End}{End}

\DeclareMathOperator{\dbc}{D_{c}^{b}}
\DeclareMathOperator{\constant}{\underline{\mathscr{C}}}
\newcommand{\cusp}{\mathbf{c}}

\begin{document}

\title[ Sheaves on the Nilpotent Cone]{Perverse Sheaves on the Nilpotent Cone and Lusztig's Generalized Springer Correspondence}

\author{Laura Rider}
\address{Department of Mathematics\\
Massachusetts Institute of Technology\\
Cambridge, Massachusetts}
\email{laurajoy@mit.edu}

\author{Amber Russell}
\address{Department of Mathematics\\
University of Georgia\\
Athens, Georgia}
\email{arussell@math.uga.edu}

\thanks{L.R. was supported by an NSF Postdoctoral Research Fellowship. A.R. was supported in part by the UGA VIGRE II grant DMS-07-38586 and the NSF RTG grant DMS-1344994 of the AGANT research group at UGA.}

\begin{abstract} In this note, we consider perverse sheaves on the nilpotent cone. We prove orthogonality relations for the equivariant category of sheaves on the nilpotent cone in a method similar to Lusztig's for character sheaves. We also consider cleanness for cuspidal perverse sheaves and the (generalized) Lusztig--Shoji algorithm.
\end{abstract}

\maketitle

\section{Introduction}

Let $G$ be a connected, reductive algebraic group defined over an algebraically closed field of good characteristic, and let $\cn$ be its nilpotent cone. We consider $\dbg(\cn)$, the $G$-equivariant derived category of constructible sheaves on $\cn$. This category encodes the representation theory of the Weyl group $W$ of $G$ via a perverse sheaf called the Springer sheaf $\mathbb{A}$ (see Section \ref{sec:prelim} for a definition) and Springer's correspondence. Of course, the category $\perv_G(\cn)$ of $G$-equivariant perverse sheaves on $\cn$ contains much more information---there are $G$-equivariant perverse sheaves on $\cn$ that are not part of the Springer correspondence. In \cite{L1}, Lusztig accounts for the extra information and relates it to the representation theory of relative Weyl groups in his generalized Springer correspondence. His classification relies on understanding the \textit{cuspidal data} associated to $G$ (see Definition \ref{def:cuspdata}).

Lusztig's work also reveals that understanding the stalks of these perverse sheaves on $\cn$ (given by generalized Green functions) is an important part of the computation of characters of finite groups of Lie type. Using the Lusztig--Shoji algorithm \cite{S, L5}, these stalks can be computed using only knowledge of Weyl (or relative Weyl) group representations.

Our primary goal is to understand Lusztig's work on the generalized Springer correspondence using methods inspired by \cite{A} in the Springer setting. One of our main results is Theorem \ref{thm:orthog}, which gives an orthogonal decomposition:  \[\dbg(\cn) \cong \bigoplus_{\cusp/\sim} \dbg(\cn, \mathbb{A}_{\cusp}).\] Here, $\cusp$ denotes a cuspidal datum and $\dbg(\cn, \mathbb{A}_{\cusp})$ is the triangulated category generated by the simple summands of the perverse sheaf  $\mathbb{A}_{\cusp}$ induced from the cuspidal datum $\cusp$. This result relies on a key property enjoyed by cuspidal local systems on the nilpotent cone in good characteristic proven by Lusztig. That is, distinct cuspidals have distinct central characters. 

In Section \ref{sec:green}, we interpret the definitions of generalized Green functions directly in terms of perverse sheaves on the nilpotent cone avoiding characteristic functions on the corresponding group. We reprove the Lusztig--Shoji algorithm in Theorem \ref{thm:LSalgo} using the re-envisioned Green functions. Our proofs are heavily influenced by those of Lusztig in \cite{L5}. From this point of view, Theorem \ref{thm:orthog} may be viewed as a categorical version of the orthogonality relations among generalized Green functions. 

In Section \ref{sec:clean}, we prove cleanness for cuspidal local systems on the nilpotent cone (see Proposition  \ref{cusp clean}). In \cite{L5}, Lusztig proves the much stronger result that cuspidal character sheaves are clean. Logically, our proof is similar in spirit to Lusztig's. However, a key role in our argument is played by the orthogonal decomposition in Theorem  \ref{thm:orthog} considerably simplifying the exposition. In particular, from this point of view, cleanness for local systems is implied by the fact that the triangulated categories generated by the corresponding simple perverse sheaf and its Verdier dual are orthogonal to the `rest' of the category.

Proposition \ref{prop:winvar} provides a computation of the $\Ext$-groups between perverse sheaves on the nilpotent cone in terms of relative Weyl group invariants. This proposition is a first step towards proving formality for non-Springer blocks of sheaves on the nilpotent cone. Formality for the other blocks would complete the description of the equivariant derived category on the nilpotent cone, as was initiated by the first author \cite{Ri} in the Springer case.

\subsection*{Organization of the paper} We review perverse sheaves on the nilpotent cone and the generalized Springer correspondence in Section \ref{sec:prelim}. In Section \ref{sec:orthog}, we prove our orthogonal decomposition (Theorem \ref{thm:orthog}). In Section \ref{sec:clean}, we prove cleanness for cuspidal local systems (\ref{cusp clean}), give some $\Ext$ computations (Proposition \ref{prop:winvar}), and discuss purity for these $\Ext$-groups (Corollary \ref{cor:pure}). In Section \ref{sec:green}, we discuss generalized Green functions and the Lusztig--Shoji algorithm. Finally, in the appendix, we review some properties of central characters for equivariant perverse sheaves. In particular, we prove Proposition \ref{ccHomvanish}, that perverse sheaves with distinct central characters have no extensions betweem them.

\section{Preliminaries}\label{sec:prelim}
We fix an algebraically closed field $k$ of positive characteristic. All varieties we consider will be defined over $k$ except in Section \ref{sec:green} when we need to employ mixed sheaves as developed in \cite{De} and \cite[Section 5]{BBD}. For an action of an algebraic group $G$ on a variety $X$, we consider the categories, denoted $\perv_G(X)\subset \dbg(X)$, of $G$-equivariant perverse sheaves on $X$ and the $G$-equivariant (bounded) derived category of sheaves on $X$. See \cite{BL} for background and definitions related to these categories. All of our sheaves will have $\ql$-coefficients. For $\mathscr{F}, \mathscr{G}\in\dbg(X)$, we let $\Hom^i(\mathscr{F}, \mathscr{G}):=\Hom_{\dbg(X)}(\mathscr{F}, \mathscr{G}[i])$. All sheaf functors are understood to be derived. We denote the constant sheaf on $X$ by $\constant_X$ or just $\constant$ when there is no ambiguity.

\subsection{The Springer correspondence}
From now on, we fix a connected, reductive algebraic group $G$ defined over $k$, and we let $\cn$ denote its nilpotent cone. We also assume that $k$ has good characteristic for $G$. We consider $G$-equivariant sheaves on $\cn$ with respect to the adjoint $G$-action. 

The nilpotent cone $\cn$ is a singular variety and has a well-studied desingularization called the Springer resolution, denoted $\mu: \widetilde{\cn}\rightarrow \cn$. Let $B$ be a Borel subgroup of $G$ with Levi decomposition $B = TU$, where $T$ is a maximal torus in $G$ and $U$ is the unipotent radical of $B$. Then, $\widetilde{\cn}=G\times^B \mathfrak{u}$, where $\mathfrak{u} = \Lie(U)$. We also note that $\widetilde{\cn}$ can be identified with the cotangent bundle of the flag variety $T^*G/B$. Hence, we have maps \begin{equation}\label{Sprmaps}\cn\stackrel{\mu}{\longleftarrow}\widetilde{\cn}\stackrel{\pi}{\longrightarrow}G/B.\end{equation} \begin{definition} The \textit{Springer sheaf}, denoted $\mathbb{A}$, is defined by \[\mathbb{A}:=\mu_!\pi^*\constant[2d],\] where $\constant$ is the constant sheaf on $G/B$ and $d=\dim{\mathfrak{u}}$.\end{definition} It is well known that the Springer sheaf $\mathbb{A}$ is a semisimple perverse sheaf (see \cite{BM1}), and it is $G$-equivariant since the Springer resolution is a $G$-map.  Furthermore, its endomorphism ring is isomorphic as an algebra to the group algebra of the Weyl group $W$ of $G$, \[\End(\mathbb{A})\cong\ql[W].\] See \cite{BM1}. This ring isomorphism allows us to link the simple summands of $\mathbb{A}$ with irreducible $W$ representations. This is known as the Springer correspondence.

\subsection{The generalized Springer correspondence and cuspidal data}
In general, not all simple ($G$-equivariant) perverse sheaves on $\cn$ occur as part of the above correspondence. The goal of Lusztig's generalized Springer correspondence \cite{L1} is to systematically identify the missing pieces and to assign some representation theoretic meaning to them. To this end, we define analogues of the maps and varieties in \eqref{Sprmaps}.

Let $P$ be a parabolic subgroup of $G$ with Levi decomposition $P = LU$. We denote by $\cn_L$ the nilpotent cone for $L$ and $\mathfrak{u} = \Lie(U)$. We consider the following $G$-varieties and $G$-maps \[\widetilde{\cn}^P:=G\times^P(\mathfrak{u} + \cn_L) \hspace{.2cm}\textup{ and }\hspace{.2cm} \mathcal{C}_P:=G\times^P\cn_L,\]\[\cn\xleftarrow{\mu_P}\widetilde{\cn}^P\stackrel{\pi_P}{\longrightarrow}\mathcal{C}_P.\] The variety $\widetilde{\cn}^P$ is called a partial resolution of $\cn$ and is studied in \cite{BM2}, where it appears with the same notation. When there is no ambiguity, we omit the subscript $P$ on the maps for simplicity of notation. Note that $\mu$ is proper and $\pi$ is smooth of relative dimension $d$, so we have $\mu_! \cong \mu_*$ and $\pi^! \cong \pi^*[2d]$, where $d$ is the dimension of $\mathfrak{u}$. We consider the following functors \begin{equation}\label{eq:functors}\mathcal{I}^G_P = \mu_!\pi^*[d] \cong \mu_*\pi^![-d],\hspace{.6cm} \mathcal{R}^G_P = \pi_*\mu^![-d], \hspace{.4cm}\textup{ and }\hspace{.4cm}\widetilde{\mathcal{R}}^G_P=\pi_!\mu^*[d],\end{equation} which we refer to as induction and restriction functors. We have adjoint pairs $(\mathcal{I}^G_P, \mathcal{R}^G_P)$ and $(\widetilde{\mathcal{R}}^G_P, \mathcal{I}^G_P)$. In the case when our parabolic is a Borel $B$, we get the usual Springer resolution diagram from \eqref{Sprmaps}. By (equivariant) induction equivalence \cite{BL}, we have an equivalence of categories $\dbg(\mathcal{C}_P)\cong\dbl(\cn_L)$ that preserves the perverse $t$-structure. Thus, often we will think of the induction (respectively, restriction) functor as having domain (respectively, codomain) $\dbl(\cn_L)$:
\[\mathcal{I}^G_P: \dbl(\cn_L)\rightarrow \dbg(\cn) \textup{ and }\mathcal{R}^G_P, \widetilde{\mathcal{R}}^G_P:  \dbg(\cn)\rightarrow \dbl(\cn_L). \]

The following theorem is due to Lusztig; see \cite[Theorem 4.4]{L2} for part of the proof in the setting of character sheaves. For a proof in the setting of more general coefficient rings (Noetherian commutative ring of finite global dimension), see \cite[Proposition 4.7]{AHR}.
 \begin{theorem}The functors $\mathcal{I}^G_P$, $\mathcal{R}^G_P$, and $\widetilde{\mathcal{R}}^G_P$ are exact with respect to the perverse $t$-structure. \end{theorem}


 \begin{definition}\label{def:cuspdata} A simple perverse sheaf $\mathcal{F}\in\dbg(\cn)$ is called \textit{cuspidal} if $\mathcal{R}^G_P(\mathcal{F})= 0$ for all proper parabolics $P$. Let $\cl$ be a local system on a nilpotent orbit $\co\subset\cn$.  Then $\cl$ is called a \textit{cuspidal local system} if $\ic(\co, \cl)$ is cuspidal.  
 
A \textit{cuspidal datum} is a tuple $\cusp = (L, \co_L, \cl)$ where $L$ is a Levi subgroup of $G$, $\co_L$ is an $L$-orbit in $\cn_L$, and $\cl$ is a cuspidal local system on $\co_L$.  For brevity, we will often denote the simple perverse sheaf which corresponds to the cuspidal datum $\cusp$ as $\ic_\cusp$.\end{definition} 

Two cuspidal data $\cusp$ and $\cusp '$ are equivalent if they are conjugate in $G$, and in this case, we write $\cusp \sim \cusp '$.

\begin{definition} For the cuspidal datum $\cusp=(G, \co_G, \cl)$, we set $\mathbb{A}_{\cusp}  = \ic_\cusp$.  
Let $\cusp=(L, \co_L, \cl)$ be a cuspidal datum for $G$. We define the perverse sheaf $\mathbb{A}_{\cusp} = \mathcal{I}^G_P(\ic_\cusp)$, and call it a \textit{Lusztig sheaf}. Let $\dbg(\cn, \mathbb{A}_{\cusp})$ be the triangulated subcategory of $\dbg(\cn)$ generated by the simple summands of $\mathbb{A}_{\cusp}$. 
\end{definition}

\begin{remark}\label{rem:dualcusp} The only cuspidal datum when the Levi is a torus $T$ is $(T, \pt, \constant)$. This datum gives rise to the Springer sheaf $\mathbb{A}$. Also, as Lusztig noted in \cite[Section 2.5]{L1}, if $\ic(\co, \cl)$ is cuspidal, then so is $\vd\ic(\co, \cl)=\ic(\co, \cl^\vee),$ where $\vd$ is the Verdier duality functor and $\cl^\vee$ is the local system dual to $\cl$. \end{remark} 

Fix a cuspidal datum $\cusp = (L, \co, \cl)$ so that $\mathbb{A}_{\cusp} = \mathcal{I}_P^G(\ic_\cusp)$, where $\ic_\cusp$ is cuspidal in $\dbl(\cn_L)$.  Let $W(L) = N_G(L)/L$. This group is referred to as a relative Weyl group. Let $\textup{Irr}(W(L))$ denote the set of isomorphism classes of irreducible $W(L)$-representations. For each $\psi\in\irr(W(L))$, we fix a representative $V_\psi$. Lusztig proves in \cite[Theorem 9.2]{L1} that we have an algebra isomorphism: \begin{equation}\label{eq:endgenspr}\End(\mathbb{A}_{\cusp})\cong\ql[W(L)].\end{equation}This, together with the fact that $\perv_G(\cn)$ is a semisimple abelian category imply that we have a decomposition of $\mathbb{A}_{\cusp}$ into isotypic components indexed by the irreducible representations of $W(L)$
\[\mathbb{A}_{\cusp} \cong \bigoplus_{\psi\in \textup{Irr}(W(L))} \ic_\psi\otimes V_\psi.\] The assignment $\ic_\psi \mapsto V^*_\psi$ is known as the generalized Springer correspondence \cite{L1}, where $V^*_\psi$ is the contragredient representation of $V_\psi$. Hence, we have an equivalence \[\perv_G(\cn) \cong \prod_{\cusp/\sim}\rep W(L_\mathbf{c}).\]

A simple $G$-equivariant perverse sheaf on $\cn$ is either cuspidal or induced from a cuspidal perverse sheaf for a proper Levi subgroup of $G$. In \cite[Theorem 4.4]{L2} Lusztig proves the statement for character sheaves, but it is easy to see that it also applies in this setting. Thus, it is sufficient to classify the cuspidal data (up to equivalence) for each group, which Lusztig does in \cite{L1}.  We use this classification to prove the following proposition.
 
 \begin{proposition}\label{prop:preser} Let $\cusp = (L, \co, \cl)$ be a cuspidal datum for $G$.  Then $\mathcal{R}^G_P\mathcal{I}^G_P\ic_\cusp$ is a direct sum of copies of $\ic_\cusp.$ Moreover, the same statement holds if $\mathcal{R}^G_P$ is replaced by $\widetilde{\mathcal{R}}^G_P$.\end{proposition}
\begin{proof} The result relies on Lusztig's generalized Springer correspondence as a classification of the simple perverse sheaves on $\cn$. 

Let $S$ be a simple summand of $\mathcal{R}^G_P\mathcal{I}^G_P\ic_\cusp$ in $\perv_L(\cn_L)$. We have that $S$ is a direct summand of a Lusztig sheaf $\mathcal{I}^L_{Q}\ic_{\cusp'}$ on $\cn_L$, where $\cusp '$ is a cuspidal datum for $L$. Note that we allow the case where $Q = L$. Hence, $\Hom(S, \mathcal{R}^G_P\mathcal{I}^G_P\ic_\cusp)\neq 0$ implies \[\Hom(\mathcal{I}^L_{Q}\ic_{\cusp '}, \mathcal{R}^G_P\mathcal{I}^G_P\ic_\cusp)\neq 0.\] Let $\tilde{Q} = QU_P$, where $U_P$ is the unipotent radical of $P$. Then $\tilde{Q}$ is a parabolic subgroup of $G$. By adjunction and the transitivity of induction $\mathcal{I}^G_{\tilde{Q}} \cong \mathcal{I}^G_{P}\mathcal{I}^L_{Q}$ \cite[Proposition 4.2]{L2}, \[\Hom(\mathcal{I}^G_{\tilde{Q}}\ic_{\cusp '}, \mathcal{I}^G_P\ic_\cusp)\neq 0.\] This implies that $\mathcal{I}^G_{\tilde{Q}}\ic_{\cusp '}$ and  $\mathcal{I}^G_P\ic_\cusp$ have a simple summand in common. However,  Lusztig's generalized Springer correspondence partitions the set of simple $G$-equivariant perverse sheaves on $\cn$ into distinct classes, each labeled by a unique cuspidal datum up to $G$-conjugation. Hence, if $\mathcal{I}^G_{\tilde{Q}}\ic_{\cusp '}$ and  $\mathcal{I}^G_P\ic_\cusp$ have a simple summand in common, then it must be that \[\mathcal{I}^G_{\tilde{Q}}\ic_{\cusp '}\cong\mathcal{I}^G_P\ic_\cusp,\] each labeled by cuspidal data that are $G$-conjugate. Since $Q\subset L$, we have $L = Q$. Thus, $S$ is a summand of (and so, must be equal to) $\mathcal{I}^L_{Q}(\ic_{\cusp '}) \cong \ic_\cusp$. 

The case for $\widetilde{\mathcal{R}}^G_P$ follows in a similar manner.
\end{proof}

\section{Orthogonal decomposition}\label{sec:orthog}

In this section, we prove our main result: an orthogonal decomposition of the category $\dbg(\cn)$ into blocks, each corresponding to a cuspidal datum for $G$. 

\begin{definition}Let $\mathscr{T}$ be a triangulated category. We say that $\mathscr{T}_1$ and $\mathscr{T}_2$, two full triangulated subcategories of $\mathscr{T}$, are \textit{orthogonal} if for all objects $\mathscr{F}\in\mathscr{T}_1$ and $\mathscr{G}\in\mathscr{T}_2$, we have that \[\Hom_{\mathscr{T}}(\mathscr{F},\mathscr{G}) =  \Hom_{\mathscr{T}}(\mathscr{G}, \mathscr{F}) = 0.\] If we have a (finite) collection $\mathscr{T}_i$ of triangulated subcategories of $\mathscr{T}$ that are pairwise orthogonal and that generate $\mathscr{T}$ such that each $\mathscr{T}_i$ cannot be split further, then we call them \textit{blocks}. The equivalence \[\mathscr{T} \cong \bigoplus_{i}\mathscr{T}_i\] is called a \textit{block decomposition} of $\mathscr{T}$.\end{definition}

Suppose we have cuspidal perverse sheaves $\ic_\cusp$ on $\cn_L$ and $\ic_{\cusp'}$ on $\cn_{L'}$, where $L$ and $L'$ are the Levi factors of parabolics $P$ and $P'$ in $G$. Let $E = \cn_{P'\cap L}\times_{\cn_{L\cap L'}}\cn_{P\cap L'}$. 
\begin{proposition}\label{propeasycohovanish}If $P'\cap L \times P\cap L'$ is properly contained within $L\times L'$, then $\cH^i_c(E, \ic_\cusp\boxtimes \ic_{\cusp'}) = 0$ for all $i$. 
\end{proposition}
\begin{proof} Consider the following diagram. 

\begin{center}
\begin{tikzcd}
\cn_L\times\cn_{L'} & \\
 \cn_{P'\cap L}\times\cn_{P\cap L'}\arrow[hookrightarrow]{u}{i}\arrow{r}{\varpi} &  \cn_{L\cap L'}\times\cn_{L\cap L'} \\ 
  E=\cn_{P'\cap L}\times_{\cn_{L\cap L'}}\cn_{P\cap L'}\arrow[hookrightarrow]{u}{s}\arrow{r}{\tilde{\varpi}} & \cn_{L\cap L'}\arrow[hookrightarrow]{u}{\Delta}
\end{tikzcd}
\end{center}

The maps $i$ and $s$ are inclusions, and $\Delta$ is diagonal inclusion.  Note that $P'\cap L$ (resp. $P\cap L'$) is a parabolic subgroup of $L$ (resp. $L'$) with Levi factor $L\cap L'$.  Thus, $\varpi$ is induced by the quotient $P'\cap L\times P\cap L'\rightarrow L\cap L'\times L\cap L'$ with unipotent kernel, $\tilde{\varpi}$ is its restriction, and the square commutes.

By base-change, we have an isomorphism $\tilde{\varpi}_! s^*\cong \Delta^*\varpi_!$. The functor $\varpi_!i^*$ is easily seen to be the non-equivariant version of restriction $\widetilde{\mathcal{R}}$ as in \eqref{eq:functors}.  Hence, $\varpi_!i^*(\ic_\cusp\boxtimes \ic_{\cusp'}) = 0$ since $\ic_\cusp\boxtimes \ic_{\cusp'}$ is cuspidal on $\cn_L\times\cn_{L'}$ and $P'\cap L \times P\cap L'$ is a proper parabolic in $L\times L'$. Hence, we have that $0 = \Delta^*\varpi_!i^*(\ic_\cusp\boxtimes \ic_{\cusp'}) = \tilde{\varpi}_!s^*i^*(\ic_\cusp\boxtimes \ic_{\cusp'})$. Therefore, $\cH^i_{c}(E, \ic_\cusp\boxtimes \ic_{\cusp'}) = 0$ for all $i$. 
\end{proof}

Let $Z = \widetilde{\cn}^P\times_{\cn}\widetilde{\cn}^{P'}$ be a generalized Steinberg variety. 
We use $\ic_\cusp$ and $\ic_{\cusp'}$ to also denote the $G$-equivariant perverse sheaves on $\cc_P$ and $\cc_{P'}$ which correspond to the cuspidal perverse sheaves on $\cn_L$ and $\cn_{L'}$ under induction equivalence.  Consider the composition \begin{equation}\label{map:tau}\tau : \widetilde{\cn}^P\times_{\cn}\widetilde{\cn}^{P'}\stackrel{i}{\hookrightarrow}\widetilde{\cn}^P\times\widetilde{\cn}^{P'}\stackrel{\Pi := \pi\times\pi'}{\longrightarrow}\cc_P\times\cc_{P'}.\end{equation} Our goal in the following proposition is to prove that the cohomology groups $\cH_G^i(Z, \tau^! \ic_{\cusp}\boxtimes \ic_{\cusp'})$ vanish for all $i$ when $\ic_{\cusp}$ and $\ic_{\cusp'}$ arise from cuspidals with non-conjugate Levis.  The outline of our proof follows that of a similar result due to Lusztig \cite[Prop. 7.2]{L3} in the setting of character sheaves. 

\begin{proposition}\label{propcohovanish} Suppose $P$ and $P'$ are parabolics in $G$ such that their Levi factors $L$ and $L'$ are non-conjugate.  Let $Z$ be as above.  Then, $\cH^i(Z, \tau^! \ic_{\cusp}\boxtimes \ic_{\cusp'}) = 0$ for all $i$. \end{proposition}

\begin{proof} We will first show an analogous statement for cohomology with compact support. Let $x\in G/P\times G/P'$, and consider the following diagram.

\begin{center}
\begin{tikzcd}
\phi^{-1}(x)\arrow{d}{\phi_x}\arrow[hookrightarrow]{r}{\iota} &  Z\arrow{d}{\phi}\\
 \{x\}\arrow[hookrightarrow]{r}{} & G/P\times G/P'\arrow{d}{}\\
  & \pt
\end{tikzcd}
\end{center} Clearly, it is enough to show that $\phi_!\tau^*( \ic_\cusp\boxtimes\ic_{\cusp'}) = 0$. This holds if all the stalks of $\phi_!\tau^*(\ic_\cusp\boxtimes\ic_{\cusp'})$ vanish. This is equivalent to $\phi_{x\,!}(\iota^*\tau^*(\ic_\cusp\boxtimes\ic_{\cusp'}))=0$ for all $x\in G/P\times G/P'$ by base change. In other words, it is enough to show that $\cH^\bullet_c(\phi^{-1}(x), \tau^*(\ic_\cusp\boxtimes\ic_{\cusp'}))=0$ for all $x$. Since $\phi_!\tau^*(\ic_\cusp\boxtimes\ic_{\cusp'})$ is $G$-equivariant, it is enough to check \[\phi_{x\, !} (\iota^*\tau^*(\ic_\cusp\boxtimes\ic_{\cusp'}))=0\] for a single $x$ in each $G$-orbit of $G/P\times G/P'$.
Let $n\in G$ be such that $L$ and $nL'n^{-1}$ share a maximal torus and consider $Z(n) = \phi^{-1}((P, nP'))$, a subvariety of $Z$.  We must show that $\cH^i_{c}(Z(n), \tau^*(\ic_{\cusp}\boxtimes \ic_{\cusp'})) = 0$ for all $i$. Consider the following commutative diagram where $L'' = nL'n^{-1}, P'' = nP'n^{-1}, f: \cn_{L''} \stackrel{\sim}{\rightarrow}\cn_{L'},$ and $E=\cn_{P''\cap L}\times_{\cn_{L\cap L''}}\cn_{P\cap L''} $.\begin{center}
\begin{tikzpicture}[description/.style={fill=white,inner sep=1.5pt}] 
\matrix (m) [matrix of math nodes, row sep=2.5em,
column sep=2em, text height=1ex, text depth=0.25ex, nodes in empty cells]
{  Z &\cc_P\times \cc_{P'} &\cc_P\times \cc_{P''}  \\ 
      & & \cn_L\times\cn_{L''}\\
      & & \cn_{P''\cap L}\times\cn_{P\cap L''}\\
      Z(n) &\cn_P\cap\cn_{P''}& E\\};
\path[->,font=\scriptsize]
(m-1-1) edge node[above]{$\tau$}(m-1-2)
 (m-1-2) edge node[above]{$\sim$}(m-1-3)
(m-4-1) edge node[above]{$\sim$} (m-4-2)
(m-4-2) edge node[above]{$\alpha$} (m-4-3);

\path[right hook->,font=\scriptsize]
(m-4-1) edge (m-1-1)
(m-2-3) edge (m-1-3)
(m-3-3) edge (m-2-3)
(m-4-3) edge (m-3-3);
\end{tikzpicture}\end{center}
Let $\ic_{\cusp''} = f^*\ic_{\cusp'}$. Hence, $\cH^i_c(Z(n), \tau^*(\ic_{\cusp}\boxtimes \ic_{\cusp'}))\cong \cH^i_c(Z(n), \alpha^*(\ic_\cusp\boxtimes \ic_{\cusp''}))$. The map $\alpha$ is smooth of relative dimension $d$ where $d= \dim U_{P\cap P''},$ hence $\alpha^*[2d]\cong\alpha^!$. 

Finally, we apply a special case of Braden's hyperbolic localization: equation (1) in \cite[Section 3]{Br}. In the diagram below, we let the multiplicative group $\mathbb{G}_m$ act on all varieties with compatible positive weights. Let $e: \{(0,0)\}\hookrightarrow E$ be the inclusion of the fixed point of this action, and let $\ell: E\rightarrow\{(0,0)\}$ be the map that sends every point to its limit. Note that Proposition \ref{propeasycohovanish} implies that $\ell_! \ic_\cusp\boxtimes \ic_{\cusp''} = 0$.  

 \begin{center}
\begin{tikzpicture}[description/.style={fill=white,inner sep=1.5pt}] 
\matrix (m) [matrix of math nodes, row sep=2.5em,
column sep=2em, text height=1ex, text depth=0.25ex, nodes in empty cells]
{   \{0\} & \{(0, 0)\}   \\ 
          \cn_P\cap\cn_{P''} & E\\};
\path[->,font=\scriptsize, >=angle 90]
(m-1-1) edge[bend right=30]  node[left]{$b$} (m-2-1)
		edge node[above]{=}(m-1-2)
(m-2-1) edge node[above]{$\alpha$}(m-2-2)
(m-2-2) edge[bend right=30] node[right]{$\ell$}(m-1-2)
(m-2-1) edge[bend right = 30] node[right]{$a$}(m-1-1);

\path[right hook->,font=\scriptsize, >=angle 90]
(m-1-2) edge[bend right=30] node[left]{$e$}(m-2-2);
\end{tikzpicture}
\end{center} Then hyperbolic localization implies that we have isomorphisms $e^!\cong \ell_!$ and $a_! \cong b^!$. Furthermore, the diagram commutes. Combining these, we see that 

\[a_!\alpha^!\cong b^!\alpha^! \cong e^! \cong \ell_!\] Now, we apply this to the cuspidal perverse sheaf $\ic_\cusp\boxtimes \ic_{\cusp''}$ to see that  \[ a_!\alpha^! \ic_\cusp\boxtimes \ic_{\cusp''} \cong \ell_! \ic_\cusp\boxtimes \ic_{\cusp''} = 0. \] Thus, we've shown $\cH^i_{c}(Z(n), \tau^*( \ic_{\cusp}\boxtimes \ic_{\cusp'})) = 0$ for all $i$ and $n$, which implies $\cH^i_{c}(Z, \tau^*(\ic_{\cusp}\boxtimes \ic_{\cusp'})) = 0$. 
Now we show that this implies $\cH^i(Z, \tau^* (\ic_{\cusp}\boxtimes \ic_{\cusp'})) = 0$ for all $i$.  If $\ic_{\cusp}$ and $\ic_{\cusp'}$ are distinct cuspidals, then $\vd \ic_{\cusp}$ and $\vd \ic_{\cusp'}$ are also distinct cuspidals by Remark \ref{rem:dualcusp}. By the above argument, $p_!\tau^* \vd \ic_{\cusp}\boxtimes \vd \ic_{\cusp'} = 0$ where $p: Z\rightarrow \pt$. Hence, $p_!\tau^* \vd(\ic_{\cusp}\boxtimes \ic_{\cusp'}) = \vd p_*\tau^! \ic_{\cusp}\boxtimes \ic_{\cusp'} = 0$, which implies that $p_*\tau^! \ic_{\cusp}\boxtimes \ic_{\cusp'} = 0.$
\end{proof}

Since forgetting equivariance commutes with sheaf functors for equivariant maps, the following corollary also holds.

\begin{corollary}\label{equivversion}
Suppose $P$ and $P'$ are parabolics in $G$ such that their Levi factors $L$ and $L'$ are non-conjugate.  Then, $\cH_G^i(Z, \tau^! \ic_{\cusp}\boxtimes \ic_{\cusp'}) = 0$ for all $i$. 
\end{corollary}


Now we come to the main theorem of this section: an orthogonal decomposition of the category $\dbg(\cn)$. The proof of this decomposition mostly follows from arguments found in \cite{L1} and \cite{L2}. 
\begin{theorem} \label{thm:orthog}We have an orthogonal decomposition indexed by cuspidal data up to equivalence \[\dbg(\cn) \cong \bigoplus_{\cusp/\sim} \dbg(\cn, \mathbb{A}_{\cusp}).\] 
\end{theorem}

\begin{proof} Let $\cusp=(L, \co, \cl)$ and $\cusp'=(L', \co', \cl')$ be two cuspidal data such that $\cusp\not\sim\cusp'$ with Lusztig sheaves $\mathbb{A}_{\cusp}$ and $\mathbb{A}_{\cusp'}$. We want to show that the two triangulated categories $\dbg(\cn, \mathbb{A}_{\cusp})$ and $\dbg(\cn, \mathbb{A}_{\cusp'})$ are orthogonal in $\dbg(\cn)$. It is sufficient to show that $\Hom^i(S, S') = \Hom^i(S', S) = 0$ for all $i\in\mathbb{Z}$ and all simple summands $S$ of $\mathbb{A}_{\cusp}$ and $S'$ of $\mathbb{A}_{\cusp'}$. 

We have two cases to consider: the Levis $L$ and $L'$ are either conjugate or not. In the second case, we apply equation (8.6.4) from \cite[Lemma 8.6.1]{CG} to yield \[\Hom^i(\mathbb{A}_{\cusp}, \mathbb{A}_{\cusp'})\cong \cH^i_{G}(Z, \tau^!(\vd\ic_\cusp\boxtimes\ic_{\cusp'})),\] which vanishes by Corollary \ref{equivversion}.


In the first case,  \[\Hom^i(\mathbb{A}_{\cusp'}, \mathbb{A}_{\cusp}) \cong \Hom^i(\ic_{\cusp'},\mathcal{R}^G_P\mathcal{I}^G_{P}\ic_\cusp).\] Here, we apply Proposition \ref{prop:preser} to see that $\mathcal{R}^G_P\mathcal{I}^G_{P}\ic_\cusp$ is a finite direct sum of copies of $\ic_\cusp$. Hence, it is sufficient to see that $\Hom^i(\ic_{\cusp'},\ic_\cusp) = 0$ for non-isomorphic cuspidals $\ic_{\cusp'}$ and $\ic_\cusp$. Lusztig proves in \cite{L1} that non-isomorphic cuspidal perverse sheaves have distinct central characters. Therefore we may apply  Proposition \ref{ccHomvanish} to complete the proof.\end{proof}

Let $u: Y\hookrightarrow\cn$ be a $G$-stable locally closed subvariety. Fix a Lusztig sheaf $\mathbb{A}_{\cusp}$. We define $\dbg(Y, \mathbb{A}_{\cusp})$ as the triangulated subcategory of $\dbg(Y)$ generated by restrictions $u^*S$ for all simple summands $S$ of $\mathbb{A}_{\cusp}$. (Note that $\dbg(Y, \mathbb{A}_{\cusp})$ may become trivial.)
\begin{corollary}\label{cor:restr_orthog} Let $Y$ be a $G$-stable locally closed subvariety of $\cn$. Then restriction to $Y$ preserves orthogonality. That is, we have an equivalence \[\dbg(Y) \cong \bigoplus_{\cusp/\sim} \dbg(Y, \mathbb{A}_{\cusp}).\]
\end{corollary}
\begin{proof} The proof is by induction on the number of $G$-orbits in $\bar{Y}/Y$. Theorem \ref{thm:orthog} implies the base case. The argument then follows without modification from the proof in \cite[Theorem 5.1]{A}.
\end{proof}

\begin{corollary} \label{cor: composition factors} Suppose $\cl$ is a local system on the orbit $\co$ of $\cn$ which appears as a composition factor of $\cH^i(\ic_\chi|_\co)$ for $\chi \in\irr W(L)$ for some Levi $L$.  Then $\ic(\co,\cl) \cong \ic_\psi$ for some $\psi\in \irr W(L)$ where $\ic_\chi$ and $\ic_\psi$ are in the same block. 
\end{corollary}

\section{Applications: cleanness and Ext computations}\label{sec:clean}
\begin{definition}
For a nilpotent orbit $\co$ of $\cn$ consider the inclusion $j:\co\hookrightarrow \cn$.  A local system $\cl$ on $\co$ is called \textit{clean} if $j_!\cl[\dim\co] \cong j_*\cl[\dim\co] \cong \ic(\co,\cl)$.
\end{definition}

The following statement is well known in the setting of perverse sheaves on the nilpotent cone and character sheaves. In good characteristic, this follows from work of Lusztig \cite{L1, L3,L4, L5} and in particular, \cite[Theorem 23.1]{L5}.  His argument uses the fact that all cuspidal local systems in good characteristic must have distinct central characters, and he goes on to show that this implies an orthogonality relationship which gives the result. See \cite[Proposition 7.9]{L2} for character sheaves on a semisimple group, for instance.  The proof given here relies on these facts as well since they were needed to prove the orthogonal decomposition of the previous section.  However, this proof also applies whenever a simple perverse sheaf and its Verdier dual are orthogonal to all other simples assuming some finiteness conditions on the variety.

\begin{proposition}\label{cusp clean} Cuspidal local systems on the nilpotent cone are clean.
\end{proposition}
\begin{proof} Let $\cl$ be a cuspidal local system on the orbit $\co$, and here, let $j: \co\hookrightarrow\cn$. By the orthogonal decomposition, $\ic(\co,\cl)$ is orthogonal to all other simple perverse sheaves on $\cn$. Let $\iota: \co'\hookrightarrow\cn$ be the inclusion of an orbit $\co'\subset\overline{\co}$ with $\co\neq\co'$. 

Suppose that $\iota^*\ic(\co, \cl)\neq 0$. We assume without loss of generality that $\co'$ is minimal among orbits in $\overline{\co}$ with this property, i.e. we assume that no orbit in $\overline{\co'}/\co'$ has this property. Then, there exists $i\in\mathbb{Z}$ and a simple perverse sheaf $S$ (a shifted local system) on $\co'$ such that $\Hom^i(\iota^*\ic(\co, \cl), S)\neq 0$. By adjunction, we have $\Hom^i(\ic(\co, \cl), \iota_*S)\neq 0$. Consider the distinguished triangle \[{^p\iota_*}S\rightarrow \iota_*S\rightarrow A\rightarrow {^p\iota_*}S[1]\] where $A\in {^pD}^{>0}$. This gives an exact sequence \[\Hom^i(\ic(\co, \cl), {^p\iota_*}S)\rightarrow\Hom^i(\ic(\co, \cl), \iota_*S)\rightarrow\Hom^i(\ic(\co, \cl),A).\] Note that $\Hom^i(\ic(\co, \cl),A) = 0$ since we assumed $\co'$ minimal, and the support of $A$ is contained within $\overline{\co'}/\co'$. In particular, this implies that ${^p}\iota_*S$ has $\ic(\co, \cl)$ as a direct summand since $\ic(\co, \cl)$ is orthogonal to all other simple perverse sheaves on $\cn$. However, the support of ${^p\iota}_*S$ is contained within $\overline{\co'}$, which is a contradiction. Therefore, it must be that the stalk of $\ic(\co,\cl)$ at any $x\in\co'$ is $0$. Thus, $j_!\cl[d] = \ic(\co,\cl)$, where $d=\dim\co$.

Recall that $\cl^\vee$ is also cuspidal on $\co$ (Remark \ref{rem:dualcusp}). The above argument proves that $j_!(\cl^\vee) [d]\cong \ic(\co, \cl^\vee)$ as well. Applying Verdier dual yields $j_*\cl [d]\cong\ic(\co, \cl)$.\end{proof}
 
\begin{remark} This result is now known to hold for cuspidal character ($\ql$-) sheaves on $G$ defined over a field of any characteristic. The final cases were considered in \cite{O} and \cite{L7}.
\end{remark}
\subsection{Some Ext computations}

For the remainder of this section, let us assume we are in a particular block of the decomposition corresponding to the cuspidal datum $\cusp = (L,\co,\cl)$.  This block will have cuspidal simple perverse sheaf $\ic_\cusp$ on $\cn_L$.  The simple summand of $\mathbb{A}_\cusp$ corresponding to the representation $V^*_\psi$ of $W(L)$ will be denoted $\ic_\psi$.  

\begin{lemma}\label{prop:extiscoho} We have that \[\Hom^i_{\dbl(\cn_L)}(\ic_\cusp, \ic_\cusp) \cong \cH^i_L(\co) \cong \cH^i_G(G\times^P\co).\] In particular, $\Hom^i_{\dbl(\cn_L)}(\ic_\cusp, \ic_\cusp)$ is pure of weight $i$ where Frobenius acts by $q^{i/2}$ and vanishes for $i$ odd.
\end{lemma}
\begin{proof} By Proposition \ref{cusp clean}, cuspidal local systems are clean, so we have that \[\Hom^i(\ic_\cusp, \ic_\cusp) =\Hom^i(j_!\cl[\dim\co], j_!\cl[\dim\co]).\] This reduces to a computation of local systems: $\Hom^i(j_!\cl, j_!\cl) = \Hom^i(\cl, j^!j_!\cl) = \Hom^i(\cl,\cl).$ On the other hand, $\Hom^i(\cl,\cl) = \cH^i(\co, R\dgHom(\cl, \cl)),$ and for rank one local systems, $R\dgHom(\cl, \cl) = \constant_{\co}$. Thus, $\Hom^i(\ic_\cusp, \ic_\cusp) \cong \cH^{i}_L(\co)$, as desired. 


We assume that all algebraic groups are split over $\mathbb{F}_q$. Let $x$ be a closed point in $\co$ fixed by Frobenius, so $\co \cong L\cdot x$. Let $S = \mathrm{Stab}_L(x)$ be the stabilizer of $x$ in $L$, $S^\circ =\mathrm{Stab}^\circ_L(x)$ its identity component, and $Z(L)^\circ$ be the identity component of the center of $L$. 

First, we prove that $\cH^\bullet_{S} (x) \cong \cH^\bullet_{Z(L)^\circ} (x)$. By \cite[Proposition 2.8]{L1}, the group $S^\circ/Z(L)^\circ$ is unipotent, which implies $\cH^\bullet_{S^\circ} (x) \cong \cH^\bullet_{Z(L)^\circ} (x).$ Thus, we have a homomorphism \[\cH^\bullet_{S^\circ} (x)^{S/S^\circ}\hookrightarrow\cH^\bullet_{S^\circ} (x) \cong \cH^\bullet_{Z(L)^\circ} (x).\] Moreover, $\pi^L_1(\co)\cong S/S^\circ$ acts trivially on $Z(L)^\circ$. Hence, the above injection is an isomorphism. Finally, we have an isomorphism $\cH^\bullet_{S}(x)\cong \cH^\bullet_{S^\circ} (x)^{S/S^\circ}$ (see, for instance, \cite[1.9(a)]{L6}).

By equivariant induction, we have $\cH _L ^{\bullet} (\co) \cong \cH ^{\bullet} _{S} ( x ),$ which is isomorphic to $\cH ^{\bullet} _{Z ( L )^{\circ}} ( x )$ by the above argument. Now, it is well known that \[\cH ^{\bullet}_{Z ( L )^{\circ}} ( x ) \cong \cH ^{\bullet} ( ( \mathbb{P} ^\infty )^{r} )\] where $r$ is the rank of the torus  $Z ( L )^{\circ}$ and that Frobenius acts by multiplication by $q$ in degree 2. 
\end{proof}

\begin{lemma} \label{lemma:restriso}Let $\ic_\psi$ be a simple summand of $\mathbb{A}_{\cusp}$ as described above. We have an isomorphism $\mathcal{R}_P^G(\ic_\psi)\cong \widetilde{\mathcal{R}}_P^G(\ic_\psi) \cong \ic_\cusp\otimes V_\psi^*$.
\end{lemma}
\begin{proof}By Proposition \ref{prop:preser}, we know $\mathcal{R}_P^G(\ic_\psi)$ is contained in the block $\dbl(\cn_L, \ic_\cusp)$.  Since our restriction functor $\mathcal{R}_P^G$ is $t$-exact, we know $\mathcal{R}_P^G(\ic_\psi)$ must be perverse. As $\ic_\cusp$ is the only simple perverse sheaf in $\dbl(\cn_L, \ic_\cusp)$, we have $\mathcal{R}^G_P(\ic_\psi) \cong \ic_\cusp\otimes\Hom(\ic_\cusp, \mathcal{R}^G_P(\ic_\psi)).$ 
Furthermore, we have $$\Hom(\ic_\cusp, \mathcal{R}^G_P(\ic_\psi))\cong \Hom(\mathbb{A}_{\cusp}, \ic_\psi) \\ \cong  \Hom(\ic_\psi\otimes V_\psi, \ic_\psi)\cong V_\psi^*.$$

Now, we need only to show that $\mathcal{R}$ and $\widetilde{\mathcal{R}}$ give isomorphic objects when restricted to the category of perverse sheaves.  Following the same reasoning as above, we have $\widetilde{\mathcal{R}}^G_P(\ic_\psi) \cong \ic_\cusp\otimes V_\psi.$ Since $W(L)$ is a Weyl group, which can be deduced from \cite[Theorem 5.9]{L0}, we have $V_\psi \cong V_\psi^*$, and the result follows.
\end{proof}

\begin{proposition} We have an isomorphism \[\Hom^i_{\dbg(\cn)}(\mathbb{A}_{\cusp}, \mathbb{A}_{\cusp})\cong \cH^i_L(\co)\otimes \ql[W(L)].\]
\end{proposition}
\begin{proof}First, using a similar argument as in the proof of Lemma \ref{lemma:restriso} and Lusztig's isomorphism \eqref{eq:endgenspr}, we have that $\mathcal{R}^G_P \mathbb{A}_{\cusp} \cong \ic_\cusp\otimes\Hom(\ic_\cusp, \mathcal{R}^G_P \mathbb{A}_{\cusp})\cong \ic_\cusp\otimes \ql[W(L)]$.

\begin{align*}\Hom^i_{\dbg(\cn)}(\mathbb{A}_{\cusp}, \mathbb{A}_{\cusp}) &\cong \Hom^i_{\dbl(\cn_L)}(\ic_\cusp,\mathcal{R}^G_P \mathbb{A}_{\cusp}) \\
&\cong \Hom^i_{\dbl(\cn_L)}(\ic_\cusp, \ic_\cusp\otimes\ql[W(L)])\\
&\cong \Hom^i_{\dbl(\cn_L)}(\ic_\cusp, \ic_\cusp)\otimes\ql[W(L)].\end{align*} Finally, we apply Lemma \ref{prop:extiscoho}, and the result follows. \end{proof}

\begin{proposition}\label{prop:winvar} For simple perverse sheaves $\ic_\psi$ and $\ic_\chi$ in the same block corresponding to $\psi, \chi\in\irr W(L)$, we have an isomorphism \[\Hom^i_{\dbl(\cn_L)}(\mathcal{R}^G_P(\ic_\psi), \mathcal{R}^G_P(\ic_\chi))\cong V_\psi\otimes \cH^i_L(\co)\otimes V_\chi^*.\]
Moreover, we have that
\begin{align*}\Hom^i_{\dbg(\cn)}(\ic_\psi, \ic_\chi) & \cong \Hom^i_{\dbl(\cn_L)}(\mathcal{R}^G_P(\ic_\psi), \mathcal{R}^G_P(\ic_\chi))^{W(L)}\\ & \cong  (V_\psi\otimes \cH^i_L(\co)\otimes V_\chi^*)^{W(L)}.\end{align*}
\end{proposition}
\begin{proof}
The first statement follows quickly from combining Lemmas \ref{prop:extiscoho} and \ref{lemma:restriso}.  For the second statement, the proof follows in the same manner as in \cite[Theorem 4.6]{A}.
\end{proof}

\begin{corollary}\label{cor:pure} Let $P_1, P_2\in\perv_G(\cn)$. Then $\Hom^i_{\dbg(\cn)}(P_1, P_2)$ is pure of weight $i$ for all even $i\in\mathbb{Z}$ and vanishes for all odd $i$.
\end{corollary}
\begin{proof}First, suppose that $P_1$ and $P_2$ are simple. If they are not in the same block, then $\Hom^i(P_1, P_2)$ vanishes by the orthogonal decomposition, Theorem \ref{thm:orthog}. If they are in the same block, then the result follows from Proposition \ref{prop:winvar} and the fact that $\cH^i_L(\co)$ is pure of weight $i$ for $i$ even and vanishes for $i$ odd. For general ($G$-equivariant) perverse sheaves on $\cn$, we use the facts that $\perv_G(\cn)$ is a semisimple category and that $\Hom^i(-, -)$ commutes with direct sum.
\end{proof}

\section{Generalized Green functions}\label{sec:green}
In this section, we talk about generalized Green functions. For this, we need to work in the mixed category. For $X$ an algebraic variety defined over $\mathbb{F}_q$, we consider the category of mixed $\ell$-adic complexes $\dm(X)$. There is a natural functor forgetting the Weil structure $\xi: \dm(X)\rightarrow\dbc(X\times_{\textup{Spec}(\mathbb{F}_q)}\textup{Spec}(\bar{\mathbb{F}}_q))$. The standard reference for the definition and properties of $\dm(X)$ is \cite[Section 5]{BBD}. 

We define $K(X)$ as the quotient of the Grothendieck group $K(\dm(X))/\sim$ where we identify isomorphism classes of simple perverse sheaves $[S_1]\sim[S_2]$ if $S_1$ and $S_2$ have the same weight and $\xi(S_1)\cong\xi(S_2)$. We fix a square root of the Tate sheaf. Then $K(X)$ has the structure of a $\mathbb{Z}[t^{1/2}, t^{-1/2}]$-module so that the action of $t$ corresponds to Tate twist: $[\mathcal{F}(i/2)] = t^{-i/2}[\mathcal{F}]$. For the rest of this section, we assume that $G$ is a split connected, reductive algebraic group defined over $\mathbb{F}_q$ and let $\cn$ denote its nilpotent cone. We also assume that $\mathbb{F}_q$ is sufficiently large so that all nilpotent orbits are non-empty. If $\co$ is a nilpotent orbit, then $K(\co)$ is the free $\mathbb{Z}[t^{1/2}, t^{-1/2}]$-module generated by classes of (weight 0) irreducible local systems on $\co$.

In what follows, we define four sets of polynomials: $p_{S, S'}, \tilde{p}_{S, S'}, \lambda_{S, S'},$ and $\omega_{S, S'}$.  For simplicity of notation, we assume throughout the section that $S$ and $S'$ are in the same block. For $\chi, \psi\in\irr W(L),$ let $\ic_\chi$ and $\ic_\psi$ be the simple perverse sheaves that are pure of weight 0 corresponding to $V_\chi^*$ and $V_\psi^*$. In this case, we will instead use the notation: $p_{\chi, \psi}(t), \lambda_{\chi, \psi}(t), \tilde{p}_{\chi, \psi}(t)$, and $\omega_{\chi, \psi}(t)$. Also, we denote by $\cl_\psi$ the local system such that $\ic_\psi|_{\co_\psi} = \cl_\psi[\dim\co_\psi](\frac{1}{2}\dim\co_\psi)$, where $\co_\psi$ is the orbit open in the support of $\ic_\psi$.

We define $p_{\chi, \psi}(t)\in\mathbb{Z}[t^{1/2}, t^{-1/2}]$ as
\begin{equation}\label{eq:restrpoly}[\ic_\chi|_{\co}] =\sum_{\{\psi | \co_\psi = \co\}}p_{\chi, \psi}(t)[\cl_\psi].\end{equation}
We also define `dual' polynomials. Let $j_\co:\co\rightarrow\cn$ be the inclusion of an orbit. Let $\tilde{p}_{\chi, \psi}(t)\in\mathbb{Z}[t^{1/2}, t^{-1/2}]$ be given by 
\begin{equation}\label{eq:shriekpoly}[j_\co^!\ic_\chi] =\sum_{\{\psi | \co_\psi = \co\}}\tilde{p}_{\chi, \psi}(t)[\cl_\psi].\end{equation} 

\begin{lemma} \label{lemma: easyprelation}The polynomials $p$ and $\tilde{p}$ satisfy the following relation: $\tilde{p}_{\chi, \psi}(t^{-1}) =  t^{\dim\co}p_{\chi^*, \psi^*}(t).$
\end{lemma}
\begin{proof} Let $\ic_\chi = \ic(\co_\chi, \cl_\chi)$. First, we have that $j^!_\co(\ic_\chi)\cong \vd j^*_{\co}\vd\ic_{\chi} = \vd (\ic(\co_\chi, \cl^\vee_\chi)|_{\co})$, where $\cl^\vee$ denotes the dual local system to $\cl$.  Furthermore, Verdier dual transforms local systems in the following way: \begin{equation}\label{}\vd (\cl(-i)) = \cl^\vee [2\dim\co](\dim\co +i).\end{equation} Hence, Verdier dual induces a morphism $\vd:K(\co)\rightarrow K(\co)$ given by $t^i[\cl] \mapsto t^{-\dim\co - i}[\cl^\vee]$. We apply $\vd$ to equation \eqref{eq:restrpoly} to get 
\begin{align*}[j^!_\co \ic_\chi] &= [\vd j^*_\co \ic(\co_\chi, \cl_\chi^*)]\\
                                                   &=\vd \sum_{\{\psi|\co_\psi = \co\}}p_{\chi^*, \psi}(t)[\cl_\psi]\\
                                                   &=\sum_{\{\psi|\co_\psi = \co\}}t^{-\dim\co}p_{\chi^*, \psi}(t^{-1})[\cl^*_\psi]\\
                                                   &=\sum_{\{\psi^*|\co_{\psi^*} = \co\}}t^{-\dim\co}p_{\chi^*, \psi^*}(t^{-1})[\cl_\psi].    \end{align*} Since the irreducible local systems (of weight 0) $[\cl_{\psi}]$ are linearly independent in $K(\co)$, we have that \[\tilde{p}_{\chi, \psi}(t) = t^{-\dim\co}p_{\chi^*, \psi^*}(t^{-1}).\] The result follows.\end{proof}
Since we know $\ic_\chi|_{\co_{\chi}} = \cl_\chi[\dim \co_\chi](\frac{1}{2}\dim \co_\chi)$ and $\ic_\chi$ vanishes off $\overline{\co_\chi}$, we have
\begin{equation} \label{eq:vanish}
 p_{\chi, \psi}(t) = \begin{cases}
 t^{-\dim \co_\chi/2} & \text{if $\chi=\psi$} \\
 0 & \text{if $\co_\psi \not\subset \overline{\co_\chi}$ or if $\co_\psi=\co_\chi\ \mathrm{with\ }\chi \neq \psi$}
\end{cases}
\end{equation}
We also define $\lambda_{\chi, \psi}(t)\in\mathbb{Z}[t^{1/2}, t^{-1/2}]$ by 
\begin{equation}\label{eq:lambdapoly}\begin{alignedat}{2}[R\Gamma_c(\co_\chi, \cl_\chi\otimes \cl^\vee_\psi)] &=\lambda_{\chi, \psi}(t)[\constant_{\pt}] &\ \ \mathrm{\ if}\  \co_\chi = \co_\psi\\
\lambda_{\chi, \psi}(t) & = 0 &\ \ \mathrm{\ if}\  \co_\chi \neq \co_\psi
\end{alignedat}\end{equation}
and $\omega_{\chi, \psi}(t)\in\mathbb{Z}[t^{1/2}, t^{-1/2}]$ by
\begin{equation}\label{eq:omegapoly}[\vd R\Hom(\ic_\chi, \ic_\psi)] =\omega_{\chi, \psi}(t)[\constant_{\pt}].\end{equation} 
Using Corollary \ref{cor:pure}, we can reformulate the definition of $\omega$ in the following way:
\begin{equation}\label{eq:omegapolyrewrite}
\begin{alignedat}{1} 
\omega_{\chi, \psi}(t) & = \sum_{i\in\mathbb{Z}}(-1)^i\dim H^{i}(\vd R\Hom(\ic_\chi, \ic_\psi)) t^{i/2}\\
&= \sum_{i\in\mathbb{Z}}(-1)^i\dim\Hom^{-i}(\ic_\chi, \ic_\psi) t^{i/2}.\\
\end{alignedat}\end{equation} 

\begin{remark} It is easy to see that our definitions of the polynomials $p, \tilde{p}, \lambda,$ and  $\omega$ can be extended to the case of simple perverse sheaves $S$ and $S'$ in different blocks. However, in this case, it is easy to see that $p_{S, S'} = \tilde{p}_{S, S'} =\omega_{S, S'} =0$ by Theorem \ref{thm:orthog} and Corollary \ref{cor: composition factors}. Moreover, the main result of this section Theorem \ref{thm:LSalgo} becomes trivial in this case. \end{remark}

\begin{lemma} \label{lem:omegasym} For any simple perverse sheaves $\ic_\chi$ and $\ic_\psi$ in $\perv_G(\cn)$, we have that $\omega_{\chi,\psi}(t) = \omega_{\psi, \chi}(t)$.  Furthermore, if $\ic_\chi$ and $\ic_\psi$ are in different blocks, then $\omega_{\chi,\psi}(t) = \lambda_{\chi,\psi}(t) = 0.$\end{lemma}
\begin{proof}If $\ic_\chi$ and $\ic_\psi$ are in different blocks, then $\omega_{\chi,\psi}(t) = \omega_{\psi,\chi}(t)=0$ follows from Theorem \ref{thm:orthog} and Corollary \ref{cor: composition factors} implies $\lambda_{\chi,\psi}(t) = \lambda_{\psi,\chi}(t)=0$. 

Now assume that $\ic_\chi$ and $\ic_\psi$ are in the same block. Proposition \ref{prop:winvar} implies \begin{align*}\Hom^{i}(\ic_\chi, \ic_\psi) &=(V_\chi\otimes\cH_L^{i}(\co)\otimes V^*_\psi )^{W(L)} \hspace{.5cm}\textup{  and }\\\Hom^{i}(\ic_{\psi}, \ic_{\chi}) &=(V^*_{\chi}\otimes \cH_L^{i}(\co)\otimes V_{\psi})^{W(L)}.\end{align*} Since $W(L)$ is a Weyl group, we have $V\cong V^*$ for any $W(L)$-representation $V$. In particular, $\dim(V_\chi\otimes\cH_L^{i}(\co)\otimes V^*_\psi )^{W(L)}=\dim (V^*_{\chi}\otimes \cH_L^{i}(\co)\otimes V_{\psi})^{W(L)}.$ Hence, using equation \eqref{eq:omegapolyrewrite} we obtain $\omega_{\chi, \psi}(t) = \omega_{\psi, \chi}(t)$. \end{proof}

The following is a refinement of \cite[Lemma 6.6]{A} to include perverse sheaves that are not self dual.
\begin{lemma}\label{lem:summationeqn} \[[\vd R\Hom(j^*_\co\ic_\chi, j^!_\co \ic_\psi)] = \sum_{\{\phi, \phi'|\co_\phi = \co_{\phi'} = \co\}}p_{\chi, \phi}(t)\lambda_{\phi, \phi'}(t)p_{\psi^*, \phi'^*}(t)[\constant_{\pt}]\]
\end{lemma}
\begin{proof} First, we note that $\vd(R\Hom(\mathscr{F}(n), \mathscr{G}(m))) = (\vd R\Hom(\mathscr{F}, \mathscr{G}))(n-m)$. Thus, $\vd R\Hom(-, -) : K(\co)\times K(\co) \rightarrow K(\co)$ is $\mathbb{Z}[t^{1/2}, t^{-1/2}]$ linear in the first variable and antilinear in the second variable with respect to the involution $t^{1/2}\mapsto t^{-1/2}$. Hence, \[[\vd R\Hom(j^*_\co\ic_\chi, j^!_\co \ic_\psi)] = \sum_{\{\phi, \phi'|\co_\phi = \co_{\phi'} = \co\}}p_{\chi, \phi}(t)\tilde{p}_{\psi, \phi'}(t^{-1})\vd [R\Hom(\cl_\phi, \cl_{\phi'})].\] Now, we apply Lemma \ref{lemma: easyprelation} to get \[[\vd R\Hom(j^*_\co\ic_\chi, j^!_\co \ic_\psi)] = t^{\dim\co}\hspace{-.5cm}\sum_{\{\phi, \phi'|\co_\phi = \co_{\phi'} = \co\}}\hspace{-.5cm}p_{\chi, \phi}(t)p_{\psi^*, \phi'^*}(t)[\vd R\Hom(\cl_\phi, \cl_{\phi'})].\] To finish the proof, it suffices to show $[\vd R\Hom(\cl_\phi, \cl_{\phi'})] = t^{-\dim\co}\lambda_{\phi, \phi'}(t)[\constant_{\pt}]$. Let $a: \co\rightarrow \pt$. Then $R\Hom(\cl_\phi, \cl_{\phi'}) \cong R\Gamma(\co, \cl^\vee_{\phi}\otimes \cl_{\phi'}) \cong a_* (\cl^\vee_{\phi}\otimes \cl_{\phi'})$. Now, we apply Verdier duality to get $a_! (\cl_{\phi}\otimes \cl^\vee_{\phi'})[2\dim\co](\dim\co).$ Hence, we have that $[\vd R\Hom(\cl_\phi, \cl_{\phi'})] = t^{-\dim\co}[R\Gamma_c(\co, \cl_{\phi}\otimes \cl^\vee_{\phi'})].$\end{proof}

We need more notation for the proof of the following result.  In particular, $\co_\phi \leq \co_{\phi'}$ means $\co_\phi\subset\overline{\co_{\phi'}}$ and $\co_\phi < \co_{\phi'}$ means $\co_\phi\subset\overline{\co_{\phi'}}\setminus \co_{\phi'}$.  We also use the fact that $\co_\phi = \co_{\phi^*}$.

\begin{remark}The following theorem is proven in \cite[Theorem 24.8]{L5}.  We include its proof here for completeness.  The inductive steps of the proof illustrate a method for computing the unknown polynomials $p$ and $\lambda$ from $\omega$ which is known as the Lusztig--Shoji algorithm.\end{remark}
\begin{theorem}\label{thm:LSalgo} \begin{enumerate} \item For all $\chi, \psi$, $p_{\chi, \psi}(t) =  p_{\chi^*, \psi^*}(t)$ and  $\lambda_{\chi, \psi}(t)=\lambda_{\psi, \chi}(t)$. \item The polynomials $p$ and $\lambda$ are the unique ones satisfying $$\omega_{\chi, \psi}(t) = \sum_{ \phi, \phi'} p_{\chi, \phi}(t)\lambda_{\phi, \phi'}(t)p_{\psi, \phi'}(t) $$ and conditions \eqref{eq:vanish} and \eqref{eq:lambdapoly}. \end{enumerate}
\end{theorem}
\begin{proof}
From Lemma \ref{lem:summationeqn} above and \cite[Lemma 6.4]{A}, we have 
\begin{align*}
[\vd R\Hom(\ic_\chi, \ic_\psi)]&= \sum_{\co\subset N} [\vd R\Hom(j^*_\co\ic_\chi, j^!_\co \ic_\psi)] \\
&=\sum_{\co\subset N} \sum_{\substack{\{\phi, \phi'|\\\co_\phi = \co_{\phi'} = \co\}}} p_{\chi, \phi}(t)\lambda_{\phi, \phi'}(t)p_{\psi^*, \phi'^*}(t)[\constant_{\pt}].
\end{align*}
This can be improved because $\lambda_{\phi, \phi'}$ is zero when $\co_\phi \neq \co_{\phi'}$.
Thus, we see that 
\begin{equation}\label{eq:useful}
\omega_{\chi, \psi}(t) = \sum_{ \phi, \phi'} p_{\chi, \phi}(t)\lambda_{\phi, \phi'}(t)p_{\psi^*, \phi'^*}(t).
\end{equation}
We will now use the above equation to prove (1), and thus, prove (2).

We prove both statements of (1) simultaneously by induction on $d=\dim \co_\psi$.  If $d<0$, we know the statement is trivially true.  Let us suppose then that $p_{\chi, \psi}(t) =  p_{\chi^*, \psi^*}(t)$ and  $\lambda_{\chi, \psi}(t)=\lambda_{\psi, \chi}(t)$ for $d\leq k-1$.  First, we will use \eqref{eq:useful} to show $\lambda_{\chi, \psi}(t)=\lambda_{\psi, \chi}(t)$ for $d = k$.   Since we know $\lambda_{\chi, \psi}(t)$ vanishes otherwise, we can assume that $\co_\chi = \co_{\psi}$.  
Recall that $p_{\phi,\phi}=t^{-\dim \co_\phi/2}$ for any $\phi$.  Now, with these conditions we have 
\begin{align*}
\omega_{\chi,\psi}(t) &= p_{\chi,\chi}(t)p_{\psi^*,\psi^*}(t)\lambda_{\chi, \psi}(t) + \sum_{ \substack{\co_\phi <\co_\chi \\ \co_{\phi'} < \co_{\psi}}} p_{\chi, \phi}(t)\lambda_{\phi, \phi'}(t)p_{\psi^*, \phi'^*}(t)\\
&=t^{-\dim \co_\chi}\lambda_{\chi, \psi}(t) +\sum_{ \substack{\co_\phi <\co_\chi \\ \co_{\phi'} < \co_{\psi}}} p_{\chi, \phi}(t)\lambda_{\phi, \phi'}(t)p_{\psi, {\phi'}}(t).
\end{align*}
By the inductive hypothesis, $p_{\psi^*, \phi'^*}(t)=p_{\psi, \phi'}(t)$ since $d=\dim \co_{\phi'} \leq k-1$.  After rearranging, we get 
 \begin{equation}\label{eq:induct}
 t^{-\dim \co_\chi}\lambda_{\chi, \psi}(t) = \omega_{\chi,\psi}(t) -\sum_{ \substack{\co_\phi <\co_\chi \\ \co_{\phi'} < \co_{\psi}}} p_{\chi, \phi}(t)\lambda_{\phi, \phi'}(t)p_{\psi, \phi'}(t),
\end{equation}
Since the right hand side of this equation is symmetric with respect to $\chi$ and $\psi$ by the inductive hypothesis and Lemma \ref{lem:omegasym}, we know the left hand side must be as well.  Thus, $\lambda_{\chi, \psi}(t)=\lambda_{\psi, \chi}(t)$ for $d = k$.

Now, we need to show $p_{\chi, \psi}(t) =  p_{\chi^*, \psi^*}(t)$ when $d=k$.  From \eqref{eq:vanish}, we may assume $\co_\psi \leq \co_\chi$.   Suppose $\alpha$ is such that $\co_\alpha = \co_\psi$, so that $\dim \co_\alpha= k$.  Again, we use \eqref{eq:useful}.  
\begin{align*}
\omega_{\alpha,\chi}(t) &= \sum_{\co_{\phi}=\co_{\alpha}} p_{\alpha,\alpha}(t)p_{\chi^*,\phi^*}(t)\lambda_{\alpha, \phi}(t) + \sum_{ \substack{\co_\phi <\co_\alpha\\ \co_{\phi} = \co_{\phi'}}} p_{\alpha, \phi}(t)\lambda_{\phi, \phi'}(t)p_{\chi^*, \phi'^*}(t)\\
&=t^{-(\dim \co_\alpha)/2}\sum_{\co_{\phi}=\co_{\alpha}}p_{\chi^*,\phi^*}(t)\lambda_{\alpha, \phi}(t) + \sum_{ \substack{\co_\phi <\co_\alpha\\ \co_{\phi} = \co_{\phi'}}} p_{\alpha, \phi}(t)\lambda_{\phi, \phi'}(t)p_{\chi, \phi'}(t).
\end{align*} 
Here, the induction hypothesis gives $p_{\chi^*, \phi'^*}(t)=p_{\chi, \phi'}(t)$ in the second sum since $d=\dim \co_{\phi'}<\dim \co_\alpha = k$. Similarly, we also have
\begin{align*}
\omega_{\chi, \alpha}(t) &= \sum_{\co_{\phi}=\co_{\alpha}} p_{\chi,\phi}(t)p_{\alpha^*,\alpha^*}(t)\lambda_{\phi, \alpha}(t) + \sum_{ \substack{\co_\phi <\co_\alpha\\ \co_{\phi} = \co_{\phi'}}}  p_{\chi, \phi}(t)\lambda_{\phi, \phi'}(t)p_{\alpha^*, \phi'^*}(t)\\
&=t^{-(\dim \co_\alpha)/2}\sum_{\co_{\phi}=\co_{\alpha}}p_{\chi,\phi}(t)\lambda_{\phi, \alpha}(t) + \sum_{ \substack{\co_\phi <\co_\alpha\\ \co_{\phi} = \co_{\phi'}}}  p_{\chi, \phi}(t)\lambda_{\phi, \phi'}(t)p_{\alpha, \phi'}(t).
\end{align*}
Keeping in mind the induction hypothesis and Lemma \ref{lem:omegasym}, we subtract these two formulas to get  \begin{equation}\label{eq:subtract} \sum_{\co_\phi = \co_\alpha}\left(p_{\chi^*,\phi^*}(t) - p_{\chi,\phi}(t)\right)\lambda_{\alpha,\phi}(t) = 0.\end{equation}
The classes of irreducible local systems on $\co$ form a basis in $K(\co)$, and therefore, the matrix $(\lambda_{\phi,\phi'})_{\co_\phi = \co_{\phi'}}$ must be nonsingular.  Thus, \eqref{eq:subtract} implies $p_{\chi^*,\phi^*}(t) = p_{\chi,\phi}(t)$.  This completes the induction argument and the proof of (1).
From \eqref{eq:induct} we get a recursive formula which implies the uniqueness of part (2) which completes the proof.
\end{proof}

\appendix
\section{Central characters}\label{centcharapp}

Throughout this section we will assume that we have a connected algebraic group $G$ acting on a variety $X$ with finitely many orbits so that the center $Z$ acts trivially. The $G$-action induces an action of the quotient $G/Z$ on $X$, and the action map $G\times X\rightarrow X$ factors as $$G\times X\stackrel{(q,id_X)}{\longrightarrow}G/Z\times X\stackrel{\alpha}{\rightarrow}X$$ where $q:G\rightarrow G/Z$ is the quotient.
\begin{proposition} The category of $G$-equivariant local systems on $G/Z$ (for the left translation action of $G$) is equivalent to the category of $Z/Z^\circ$-representations as a tensor category.
\end{proposition} 
\begin{proof} By induction equivalence, the category of $G$-equivariant sheaves on $G/Z$ is equivalent to the category of $Z$-equivariant sheaves on a point. Furthermore, the category of $Z$-equivariant sheaves on a point is equivalent to the category of $Z/Z^\circ$ representations. (The $Z$-equivariant fundamental group of a point is $Z/Z^0$.) \end{proof} For a character $\chi$ of $Z/Z^\circ$, we denote the corresponding local system under this equivalence by $\cl_{\chi}$.  
\begin{definition}We say that a $G$-equivariant perverse (or constructible) sheaf $\mathcal{F}$ admits a central character $\chi$ if there is an isomorphism $\alpha^*\mathcal{F}\cong \cl_{\chi}\boxtimes\mathcal{F}.$\end{definition} 
\begin{lemma} A simple $G$-equivariant perverse sheaf admits a central character.\end{lemma}
\begin{proof} Let $P = \ic(\co, \cl)$ be a simple $G$-equivariant perverse sheaf with $\co$ a $G$-orbit in $X$. The $G$-equivariant fundamental group of $\co$ is $\pi^G_1(\co) = \textup{Stab}_{G}(x)/\textup{Stab}_{G}(x)^\circ$, where $x\in\co$. Since $Z$ acts trivially on $X$, we have a map $Z/Z^\circ\rightarrow \pi^G_1(\co)$. Thus, if $\cl$ is a local system on $\co$ corresponding to the $\pi^G_1(\co)$ representation $V_\cl$, the central character of $\cl$ is the local system $\cl_\chi$ corresponding to the restriction $V_\cl|_{Z/Z^\circ}$.  To see that $\cl_\chi$ is the central character of $P$, use base change with the diagram 

\begin{center}
\begin{tikzcd}
G/Z\times \co\arrow{r}{\alpha'}\arrow[hookrightarrow]{d}{j'} & \co\arrow[hookrightarrow]{d}{j}\\
 G/Z\times X\arrow{r}{\alpha} & X
 \end{tikzcd}
 \end{center}
 
Under base change, the natural map $\alpha^*j_!\cl \rightarrow \alpha^*j_*\cl$ goes to  $j'_!\alpha'^*\cl\cong j'_!(\cl_\chi\boxtimes\cl)\rightarrow j'_*\alpha'^*\cl\cong j'_*(\cl_\chi\boxtimes\cl)$. Hence, $\alpha^*P \cong \cl_\chi\boxtimes P$.
 \end{proof}

We record some facts about external tensor product for our convenience.
\begin{enumerate}
\item $t$-exact
\item $\vd \mathcal{F}\boxtimes\vd\mathcal{G}\cong \vd(\mathcal{F}\boxtimes\mathcal{G})$
\item $(f\times g)^*(\mathcal{F}\boxtimes\mathcal{G})\cong f^*\mathcal{F}\boxtimes g^*\mathcal{G}$
\item $(f\times g)_*(\mathcal{F}\boxtimes\mathcal{G})\cong f_*\mathcal{F}\boxtimes g_*\mathcal{G}$
\item $(f\times g)_!(\mathcal{F}\boxtimes\mathcal{G})\cong f_!\mathcal{F}\boxtimes g_!\mathcal{G}$
\end{enumerate}

\begin{lemma} \label{lemma:times}Suppose $\mathcal{F}$ has central character $\chi$ and $\mathcal{G}$ has central character $\psi$. Then $\mathcal{F}\otimes^L\mathcal{G}$ has central character $\chi\otimes\psi$\end{lemma}
\begin{proof}
\begin{align*}\alpha^*(\mathcal{F}\otimes^L\mathcal{G}) &\cong\alpha^*\mathcal{F}\otimes^L\alpha^*\mathcal{G}\\
											       &\cong(\cl_{\chi}\boxtimes\mathcal{F})\otimes^L(\cl_{\psi}\boxtimes\mathcal{G})\\
											       &\cong p_{G/Z}^*\cl_{\chi}\otimes p_X^*\mathcal{F}\otimes p_{G/Z}^*\cl_{\psi}\otimes p_X^*\mathcal{G}\\ 
											       &\cong p_{G/Z}^*(\cl_{\chi}\otimes\cl_{\psi})\otimes p_X^*(\mathcal{F}\otimes\mathcal{G})\\
											       &\cong(\cl_{\chi}\otimes\cl_{\psi})\boxtimes (\mathcal{F}\otimes\mathcal{G})\qedhere
											       \end{align*}
											       \end{proof}

\begin{lemma}\label{lemma:vd}Suppose  $\mathcal{F}$ has central character $\chi$. Then $\vd\mathcal{F}$ has central character $\chi^{-1}$.\end{lemma}
\begin{proof} Since $\alpha$ is flat with $d=\dim G/Z$ dimensional fibers, we have $\alpha^!\cong\alpha^*[2d]$. Thus, we see that 
\begin{align*}\alpha^*\vd\mathcal{F}&\cong\alpha^!\vd\mathcal{F}[-2d]\\
						    &\cong\vd\alpha^*\mathcal{F}[-2d]\\
						&\cong\vd(\cl_{\chi}\boxtimes\mathcal{F})[-2d]\\ 
						 &\cong\vd\cl_{\chi}\boxtimes\vd\mathcal{F}[-2d]\\
					        &\cong\cl^{\vee}_{\chi}\boxtimes\vd\mathcal{F}[2d-2d]\\
					        &\cong\cl_{\chi^{-1}}\boxtimes\vd\mathcal{F}\qedhere
											       \end{align*}\end{proof}

\begin{lemma}\label{lemma:pullback} Suppose that $f:X\rightarrow Y$ is a $G$-equivariant map between $G$-varieties on which $Z$ acts trivially. Then $f^!$, $f^*$, $f_*$, and $f_!$ preserve central character.\end{lemma}
\begin{proof}We consider the commutative diagram:
\begin{center}
\begin{tikzpicture}[description/.style={fill=white,inner sep=1.5pt}] 

\matrix (m) [matrix of math nodes, row sep=2.5em,
column sep=2em, text height=1ex, text depth=0.25ex, nodes in empty cells]
{  G/Z\times X & G/Z\times Y\\
   X & Y\\ };
\path[->,font=\scriptsize]
(m-1-1) edge node[above]{$id\times f$}(m-1-2)
edge node[right] {$\alpha_X$} (m-2-1)
(m-1-2) edge node[right]{$\alpha_Y$} (m-2-2)
(m-2-1) edge node[above]{$f$} (m-2-2);
\end{tikzpicture}
\end{center} First, we will show pull-backs preserve central character. Suppose $\mathcal{F}$ is a $G$-equivariant sheaf on $Y$ with central character $\chi$. Then we have
\begin{align*}\alpha^*f^*\mathcal{F}&\cong(id\times f)^*\alpha^*\mathcal{F}\\
						    &\cong(id\times f)^*(\cl_{\chi}\boxtimes\mathcal{F})\\
						&\cong\cl_{\chi}\boxtimes f^*\mathcal{F}\\ 
											     \end{align*}
											     
\begin{align*}\alpha^*f^!\mathcal{F}&\cong(id\times f)^!\alpha^*\mathcal{F}\\
						    &\cong(id\times f)^!(\cl_{\chi}\boxtimes\mathcal{F})\\
						&\cong\cl_{\chi}\boxtimes f^!\mathcal{F}\\ 
											     \end{align*}	
											     
By base change and flatness of the action maps $\alpha$, we have that 
\begin{align*}\alpha^*f_*\mathcal{F}&\cong(id\times f)_*\alpha^*\mathcal{F}\\
						    &\cong(id\times f)_*(\cl_{\chi}\boxtimes\mathcal{F})\\
						&\cong\cl_{\chi}\boxtimes f_*\mathcal{F}\\ 
											     \end{align*}	
Similarly, \begin{align*}\alpha^*f_!\mathcal{F}&\cong(id\times f)_!\alpha^*\mathcal{F}\\
						    &\cong(id\times f)_!(\cl_{\chi}\boxtimes\mathcal{F})\\
						&\cong\cl_{\chi}\boxtimes f_!\mathcal{F}\qedhere
											     \end{align*}\end{proof}

\begin{lemma}\label{ccHom}Suppose $\mathcal{F}$ has central character $\chi$ and $\mathcal{G}$ has central character $\psi$. Then the sheaf $R\dgHom(\mathcal{F}, \mathcal{G})$ has central character $\chi^{-1}\otimes\psi$.\end{lemma}
\begin{proof} We have an isomorphism \[R\dgHom(\mathcal{F}, \mathcal{G}) \cong \vd(\mathcal{F}\otimes \vd\mathcal{G})).\] Of course, $\mathcal{F}\otimes \vd\mathcal{G}$ has central character $\chi\otimes\psi^{-1}$ by Lemmas \ref{lemma:vd} and \ref{lemma:times}, and we apply Lemma \ref{lemma:vd} again to get the result.\end{proof}

\begin{proposition} \label{ccHomvanish}Suppose that two perverse sheaves on $X$ have distinct central characters. Then $R\Hom(\mathcal{F}, \mathcal{F}') = 0$. In particular, $\Hom^i(\mathcal{F}', \mathcal{F}) =\Hom^i(\mathcal{F}, \mathcal{F}') =0$ for all $i$.
\end{proposition}
\begin{proof} First, we note that on $\pt$, the only possible central character is trivial. Let $\mathcal{F}$ and $\mathcal{F}'$ be perverse sheaves on $X$ with central characters $\chi$ and $\psi$ (and $\chi\neq\psi$). Let $a:X\rightarrow \pt$ be the constant map. Recall that $R\Hom(\mathcal{F}, \mathcal{F}') = a_*R\dgHom(\mathcal{F}, \mathcal{F}')$. If $R\dgHom(\mathcal{F}, \mathcal{F}')\neq 0$, then Lemma \ref{ccHom} implies that $R\dgHom(\mathcal{F}, \mathcal{F}')$ has nontrivial central character equal to $\chi^{-1}\otimes\psi$. Then,   
$a_*R\dgHom(\mathcal{F}, \mathcal{F}')$ also has central character $\chi^{-1}\otimes\psi$ by Lemma \ref{lemma:pullback} which is impossible.\end{proof}


\section*{Acknowledgements}
We are very grateful to Pramod Achar for guidance in the early stages of this work and Jay Taylor for a number of useful comments and suggestions. We would also like to thank Syu Kato for explaining the action of Frobenius in Lemma \ref{prop:extiscoho}.


\end{document}